\documentclass[reqno]{amsart}
\usepackage{amsmath, amsfonts, amsthm, amssymb, bbm}
\newtheorem{theorem}{Theorem}[section]

\newtheorem{corollary}[theorem]{Corollary}

\theoremstyle{definition}

\theoremstyle{remark}

\numberwithin{equation}{section}

\newcommand{\spt}[1]{\mbox{\normalfont spt}\Parans{#1}}
\newcommand{\sptBar}[2]{\overline{\mbox{\normalfont spt}}_{#1}\Parans{#2}}
\newcommand{\sptBarTwo}[2]{\overline{\mbox{\normalfont spt2}}_{#1}\Parans{#2}}
\newcommand{\Mspt}[2]{\mbox{\normalfont M2spt}_{#1}\Parans{#2}}

\newcommand{\Parans}[1]{\left(#1\right)}

\newcommand{\SBrackets}[1]{\left[#1\right]}

\newcommand{\PieceTwo}[4]
{
	\left\{
   	\begin{array}{ll}
      	#1 & #3 \\
       	#2 & #4
     	\end{array}
	\right.
}
\newcommand{\aqprod}[3]{\Parans{#1;#2}_{#3}}

\newcommand{\SB}{\overline{\mbox{\rm S}}}
\newcommand{\SBTwo}{\overline{\mbox{\rm S2}}}
\newcommand{\STwo}{\mbox{\rm S2}}
\newcommand{\Bin}[2]
{
	\left(
   	\begin{array}{c}
      	#1\\
       	#2
     	\end{array}
	\right)
}

\newcommand{\Floor}[1]{\lfloor #1 \rfloor}

\author{CHRIS JENNINGS-SHAFFER}
\address{Department of Mathematics, University of Florida\\
Gainesville, Florida 32611, USA
\endgraf cjenningsshaffer@ufl.edu}

\title{Higher order SPT functions for overpartitions, overpartitions with
smallest part even, and partitions with smallest part even and without repeated odd parts}

\allowdisplaybreaks
\begin{document}

\allowdisplaybreaks
\keywords{
Number theory,
partitions,
overpartitions,
rank moments,
crank moments,
Andrews' spt-function,
smallest parts function,
higher order spt functions
}

\begin{abstract}
We consider the symmetrized moments of three ranks and cranks, 
similar to the work of Garvan 
in \cite{Garvan2} for the rank and crank of a partition. By using Bailey
pairs and elementary rearrangements, we are able to find useful expressions
for these moments. We then deduce inequalities between the corresponding
ordinary moments. In particular we prove that the crank moment for overpartitions
is always larger than the rank moment for overpartitions,
$\overline{M}_{2k}(n) > \overline{N}_{2k}(n)$;
with recent asymptotics this was known to hold for sufficiently large values of
$n$ for each fixed $k$.
Lastly we provide higher order spt functions
for overpartitions, overpartitions with smallest part even, and
partitions with smallest part even and no repeated odds.
\end{abstract}

\maketitle

\section{Introduction}

\allowdisplaybreaks

Here we consider certain rank and crank moments for partition like functions
and how they relate to smallest parts functions. In particular we prove
inequalities between certain moments and define higher order smallest parts
functions as the difference of symmetrized moments.

We begin by looking at the ordinary partition function.
We recall a partition of $n$ is a non-increasing sequence of positive integers
that sum to $n$. We denote the number of partitions of $n$ by $p(n)$.
We see $p(4)=5$ since the partitions of $4$ are
$4$, $3+1$, $2+2$, $2+1+1$ and $1+1+1+1$.
In \cite{Andrews} Andrews defined $\spt{n}$ to be the total number of occurrences of
the smallest parts in the partitions of $n$.
Thus from the partitions of $4$, we see $\spt{4}=10$.

We recall the rank of a partition is the largest parts minus the number 
of parts. The crank of a partition is the largest part if there are no ones and
otherwise is the number of parts larger than the number of ones 
minus the number of ones. The first point of interest of the rank and crank of
a partition is that the rank gives a combinatorial explanation of the well known
congruences $p(5n+4)\equiv 0\pmod{5}$ and $p(7n+5)\equiv 0\pmod{7}$
and the crank gives a combinatorial explanation of 
$p(5n+4)\equiv 0\pmod{5}$, $p(7n+5)\equiv 0\pmod{7}$,
and $p(11n+6)\equiv 0\pmod{11}$. 
Specifically, if we group the partitions of $5n+4$ according to their rank 
(or crank) modulo $5$, we get $5$ equally sized sets and so
$p(5n+4)\equiv 0\pmod{5}$. Similarly if we group the partitions of $7n+5$ 
according to their rank (or crank) modulo $7$, we get $7$ equally sized sets.
If we group the partitions of $11n+6$ according to their rank 
modulo $11$ we do not in general get $11$ equally sized sets, however we do
get $11$ equally sized sets if we group by the crank modulo $11$.
As we will see shortly, the rank and crank have other uses as well.
We let $N(m,n)$ denote the number of partitions of
$n$ with rank $m$ and $M(m,n)$ denote the number of partitions of $n$
with crank $m$. After suitably altering the interpretations for $n=0$ and $n=1$, 
one has that
\begin{align*}
	C(z,q)
	&=
	\sum_{n=0}^\infty\sum_{m=-\infty}^\infty M(m,n)z^mq^n
	=
	\frac{\aqprod{q}{q}{\infty}}{\aqprod{zq,z^{-1}q}{q}{\infty}}
.
\end{align*}
For the rank we have
\begin{align*}
	R(z,q)
	&=
	\sum_{n=0}^\infty\sum_{m=-\infty}^\infty N(m,n)z^mq^n
	=
	\sum_{n=0}^\infty \frac{q^{n^2}}{\aqprod{zq,z^{-1}q}{q}{n}}
	\\
	&=
	\frac{1}{\aqprod{q}{q}{\infty}}
	\SBrackets{1 + \sum_{n=1}^\infty \frac{(1-z)(1-z^{-1})(-1)^n q^{n(3n+1)/2}(1+q^n)}
		{(1-zq^n)(1-z^{-1}q^n)}
	}
.
\end{align*}
Here and throughout the rest of this paper we are using the standard product
notation,
\begin{align*}
	\aqprod{a}{q}{n} 
	&=
	\prod_{j=0}^{n-1} (1-aq^j),
	\\
	\aqprod{a}{q}{\infty} 
	&=
	\prod_{j=0}^{\infty} (1-aq^j),
	\\
	\aqprod{a_1,\dots,a_k}{q}{n}
	&=
	\aqprod{a_1}{q}{n}\dots\aqprod{a_k}{q}{n},
	\\
	\aqprod{a_1,\dots,a_k}{q}{\infty}
	&=
	\aqprod{a_1}{q}{\infty}\dots\aqprod{a_k}{q}{\infty}
.
\end{align*}

We can now introduce the rank and crank moments
\begin{align*}
	N_k(n) &= \sum_{m=-\infty}^\infty m^k N(m,n),
	\\
	M_k(n) &= \sum_{m=-\infty}^\infty m^k M(m,n)
.
\end{align*}
Both of these sums are actually finite since $N(m,n)=M(m,n)=0$ for 
$|m|>n$. Also the odd moments are zero since $N(-m,n)=N(m,n)$ and
$M(-m,n)=M(m,n)$.
These moments were first considered by Atkin and Garvan in \cite{AG}.
By Andrews \cite{Andrews} $\spt{n} = np(n) - \frac{1}{2}N_2(n)$ and
by Dyson \cite{Dyson} $np(n) = \frac{1}{2}M_2(n)$, thus
\begin{align*}
	\spt{n} &= \frac{1}{2}M_2(n) - \frac{1}{2}N_2(n)
.
\end{align*}
We then see a useful way to study smallest parts functions is to consider the related
rank and crank moments.
Rather than immediately working with these moments, it has proved fruitful
to consider a symmetrized version (for examples of this see
\cite{Andrews2}, \cite{BLO1}, \cite{DixitYee}, \cite{Garvan2},  and
\cite{Mao}).
In \cite{Garvan2} Garvan used symmetrized moments of the 
rank and crank functions given by
\begin{align*}
	\eta_k(n) 
	&= 
	\sum_{m\in\mathbb{Z}} \Bin{m+\Floor{\frac{k-1}{2}}}{k} N(m,n),
	\\
	\mu_k(n) 
	&= 
	\sum_{m\in\mathbb{Z}} \Bin{m+\Floor{\frac{k-1}{2}}}{k} M(m,n)
,
\end{align*}
to define a higher order analog of the spt function given by
\begin{align*}
	\mbox{spt}_k(n) = \mu_{2k}(n) - \eta_{2k}(n)
.
\end{align*}
One can use $N(-m,n)=N(m,n)$ and $M(-m,n)=M(m,n)$ with the proof of 
Theorem 1 of \cite{Andrews2} to find that the odd symmetrized moments are also
zero.
In Theorem 4.3 of \cite{Garvan2}, Garvan found the following formulas relating 
the ordinary and symmetrized moments,
\begin{align*}
	\eta_{2k}(n) 
	&= 
	\frac{1}{(2k)!}
	\sum_{m\in\mathbb{Z}} g_k(m) N(m,n),
	\\
	\mu_{2k}(n) 
	&=
	\frac{1}{(2k)!} 
	\sum_{m\in\mathbb{Z}} g_k(m) M(m,n),
	\\
	N_{2k}(n) 
	&= 
	\sum_{j=1}^{k} (2j)! S^*(k,j) \eta_{2j}(n),
	\\
	M_{2k}(n) 
	&= 
	\sum_{j=1}^{k} (2j)! S^*(k,j) \mu_{2j}(n)
.
\end{align*}
Here 
\begin{align*}
	g_k(x) &= \prod_{j=0}^{k-1}(x^2-j^2)
\end{align*}
and the sequence $S^*(n,k)$ is defined recursively by
$S^*(n+1,k)=S^*(n,k-1)+k^2S^*(n,k)$ and the boundary conditions
$S^*(1,1)=1$, $S^*(n,k)=0$ for $k\le 0$ or $k>n$.

In particular we have $\mu_2(n) = \frac{1}{2}M_2(n)$ and
$\eta_2(n) = \frac{1}{2}N_2(n)$ so that $\mbox{spt}_1(n) = \spt{n}$. The
functions $\mbox{spt}_k(n)$ are interesting on their own, as they have a 
clever combinatorial interpretation and satisfy various congruences. 
Also in developing the functions $\mbox{spt}_k(n)$ we get information on
the rank and crank moments. 

In \cite{BKR} Bringmann, Mahlburg, and Rhoades derived asymptotics for
$M_{2k}$, $N_{2k}$, and $M_{2k}-N_{2k}$. In particular this showed that
for each $k\ge 1$, for sufficiently large $n$ one has $M_{2k}(n)-N_{2k}(n)>0$.
In considering the symmetrized moments that lead to $\mbox{spt}_k(n)$,
Garvan in \cite{Garvan2} proved that indeed $M_{2k}(n)-N_{2k}(n)>0$
for all $k\ge 1$ and all $n\ge 1$. 

In \cite{BKR2} Bringmann, Mahlburg, and Rhoades derived asymptotics for
$M^+_{k}$, $N^+_{k}$, and $M^+_{k}-N^+_{k}$, where
\begin{align*}
	N^+_k(n) &= \sum_{m=1}^\infty m^k N(m,n),
	\\
	M^+_k(n) &= \sum_{m=1}^\infty m^k M(m,n)
.
\end{align*}
We note $N_{2k}(n) = 2N^+_{2k}(n)$ and $M_{2k}(n) = 2M^+_{2k}(n)$.
In \cite{ACK} Andrews, Chan, and Kim deduced 
$M^+_{k}(n)>N^+_{k}(n)$ for all $k\ge 1$ and all $n\ge 1$.

We apply the idea of a higher order spt function to three other smallest parts functions. 
We use $\sptBar{}{n}$, the number of smallest parts in the overpartitions
of $n$, $\sptBarTwo{}{n}$ the number of smallest parts in the overpartitions
of $n$ with smallest part even, and $\Mspt{}{n}$ the number of smallest parts
in the partitions of $n$ with smallest part even and without repeated odd parts.

For each function we consider a
certain partition type function
and introduce a rank and crank. The smallest parts function will agree with
the difference of the second symmetrized moments of the crank and rank. 
We define a higher order smallest parts function by the difference of the
symmetrized moments.
The purpose of this paper is to find expressions for
the rank and crank moments, define the higher order smallest parts functions,
and deduce the higher order smallest parts functions
are non-negative. From this non-negativity we prove inequalities between
the ordinary crank and rank moments.
In Section 2 we give the statements and proofs of our Theorems.
We give combinatorial interpretations of the higher order smallest parts 
functions in Section 3.
In Section 4 we prove two congruences for $\sptBar{2}{n}$, a higher order
analog of $\sptBar{}{n}$.
All the machinery from \cite{Garvan2} can be reused for these 
purposes. In the following subsections we discuss each of the three smallest 
parts functions.

\subsection{The Number of Smallest Parts in Overpartitions}

An overpartition of $n$ is a partition of $n$ in which the first
occurrence of a part may be overlined. We denote the number of overpartitions
of $n$ by $\overline{p}(n)$. Thus while $p(4)=5$ we have instead
$\overline{p}(n)=14$ since the overpartitions of $4$
are $4$, $\overline{4}$, $3+1$, $3+\overline{1}$, $\overline{3}+1$, 
$\overline{3}+\overline{1}$, $2+2$, $\overline{2}+2$, $2+1+1$,
$2+\overline{1}+1$, $\overline{2}+1+1$, $\overline{2}+\overline{1}+1$,
$1+1+1+1$, and $\overline{1}+1+1+1$.

In \cite{BLO2} Bringmann, Lovejoy, and Osburn defined  $\sptBar{}{n}$ 
as the number of 
smallest parts in the overpartitions of $n$. 
We use the convention of only including the
overpartitions where the smallest part is not overlined.
We see then $\sptBar{}{4}=13$.

As in \cite{BLO2} and others, for an overpartition $\pi$ of $n$ we define the 
Dyson rank of $\pi$ to be the largest part minus the number of parts of $\pi$. 
We let $\overline{N}(m,n)$ denote the number of overpartitions of $n$ with Dyson rank 
equal to $m$. As in Proposition 1.1 and the proof of Proposition 3.2 of 
\cite{Lovejoy1}, the generating function for $\overline{N}(m,n)$ is given by
\begin{align}
	\overline{R}(z,q)
	&=
	\sum_{n=0}^\infty\sum_{m=-\infty}^\infty \overline{N}(m,n)z^mq^n
	= \sum_{n=0}^\infty\frac{\aqprod{-1}{q}{n}q^{n(n+1)/2}}
		{\aqprod{zq}{q}{n}\aqprod{z^{-1}q}{q}{n}}
	\\\label{OverpartitionDysonRank}
	&= \frac{\aqprod{-q}{q}{\infty}}{\aqprod{q}{q}{\infty}}
		\SBrackets{1+
			2\sum_{n=1}^\infty\frac{(1-z)(1-z^{-1})(-1)^nq^{n^2+n}}
				{(1-zq^n)(1-z^{-1}q^n)}
		}
.
\end{align}
The second equality is an application Watson's transformation. We 
recall Watson's transformation is
\begin{align*}
	&\sum_{n=0}^\infty \frac{\aqprod{aq/bc,d,e}{q}{n}(\frac{aq}{de})^n}
			{\aqprod{q,aq/b,aq/c}{q}{n}}
	\\	
	&=
	\frac{\aqprod{aq/d,aq/e}{q}{\infty}}{\aqprod{aq,aq/de}{q}{\infty}}
	\sum_{n=0}^\infty 
	\frac{\aqprod{a,\sqrt{a}q,-\sqrt{a}q,b,c,d,e}{q}{n}(aq)^{2n}(-1)^nq^{n(n-1)/2}}
		{\aqprod{q,\sqrt{a},-\sqrt{a},aq/b,aq/c,aq/d,aq/e}{q}{n}(bcde)^n}
.
\end{align*}

As in \cite{BLO2}, for an overpartition $\pi$ of $n$ we define a residual crank 
of $\pi$ by the crank of the subpartition of $\pi$ consisting of the non-overlined 
parts of $\pi$. We let $\overline{M}(m,n)$ denote the number of overpartitions of 
$n$ with this residual crank equal to $m$. The generating function for 
$\overline{M}(m,n)$ is then given by
\begin{align}\label{OverpartitionResidualCrankDef}
	\sum_{n=0}^\infty\sum_{m=-\infty}^\infty \overline{M}(m,n)z^mq^n
	&= \frac{\aqprod{-q}{q}{\infty}\aqprod{q}{q}{\infty}}
		{\aqprod{zq}{q}{\infty}\aqprod{z^{-1}q}{q}{\infty}}.
\end{align}
Of course this interpretation is not quite correct, as  
$\frac{\aqprod{q}{q}{\infty}}{\aqprod{zq,z^{-1}q}{q}{\infty}}$
does not agree at $q$ for the crank of the partition consisting of a single one. Thus
the interpretation of this residual crank is not quite correct for overpartitions
whose non-overlined parts consist of a single one. However, this is the generating
function we must use.

We have the ordinary and symmetrized moments defined by
\begin{align*}
	\overline{N}_k(n) 
	&= 
	\sum_{m\in\mathbb{Z}} m^k \overline{N}(m,n),
	\\
	\overline{M}_k(n) 
	&= 
	\sum_{m\in\mathbb{Z}} m^k \overline{M}(m,n),
	\\
	\overline{\eta}_k(n) 
	&= 
	\sum_{m\in\mathbb{Z}}\Bin{m+\Floor{\frac{k-1}{2}}}{k} \overline{N}(m,n),
	\\
	\overline{\mu}_k(n) 
	&= 
	\sum_{m\in\mathbb{Z}} \Bin{m+\Floor{\frac{k-1}{2}}}{k} \overline{M}(m,n).
\end{align*}
Again these sums are actually finite sums
and the odd moments are zero due 
to the symmetry $\overline{N}(-m,n)=\overline{N}(m,n)$ and
$\overline{M}(-m,n)=\overline{M}(m,n)$.
We find the proof of Theorem 4.3 of
\cite{Garvan2} works to give that
\begin{align*}
	\overline{\eta}_{2k}(n) 
	&= 
	\frac{1}{(2k)!}
	\sum_{m\in\mathbb{Z}} g_k(m) \overline{N}(m,n),
	\\
	\overline{\mu}_{2k}(n) 
	&= 
	\frac{1}{(2k)!}
	\sum_{m\in\mathbb{Z}} g_k(m) \overline{M}(m,n),	
	\\
	\overline{N}_{2k}(n) 
	&= 
	\sum_{j=1}^{k} (2j)! S^*(k,j) \overline{\eta}_{2j}(n)
	,
	\\
	\overline{M}_{2k}(n) 
	&= 
	\sum_{j=1}^{k} (2j)! S^*(k,j) \overline{\mu}_{2j}(n)
.
\end{align*}

Although it is not immediately apparent, similar to $\spt{n}$ we do have 
$\sptBar{}{n} = \overline{\mu}_{2}(n) - \overline{\eta}_{2}(n)$.
We then define the higher order spt function 
$\sptBar{k}{n} = 	\overline{\mu}_{2k}(n) - \overline{\eta}_{2k}(n)$. That
$\sptBar{}{n}$ is indeed the difference of the symmetrized moments 
follows by the combinatorial interpretation of the higher order 
$\sptBar{k}{n}$ in Section 3.

In Corollary \ref{CorollaryForSptBar} we find $\sptBar{k}{n}$ has the 
generating function
\begin{align*}
	\sum_{n=1}^\infty \sptBar{k}{n} q^n
	&=
	\sum_{n=1}^\infty (\overline{\mu}_{2k}(n)- \overline{\eta}_{2k}(n)) q^n
	\\
	&=
	\sum_{n_k\ge n_{k-1}\ge \dots \ge n_1 \ge 1}
	\frac{q^{n_1+n_2+\dots+n_k}}{(1-q^{n_k})^2(1-q^{n_{k-1}})^2\dots(1-q^{n_1})^2}
	\frac{\aqprod{-q^{n_1+1}}{q}{\infty}}{\aqprod{q^{n_1+1}}{q}{\infty}}
.
\end{align*}
In Corollary \ref{CorollaryMomentInequalities} we use this
to prove the inequality 
$\overline{M}_{2k}(n) > \overline{N}_{2k}(n)$ for  all $k\ge 1$ and all $n\ge 1$. 
Previously this inequality was known to hold for each fixed $k$
for sufficiently large $n$, due to the work of
Zapata Rolon \cite{Zapata} in determining the asymptotics
for $\overline{M}^+_{k}$, $\overline{N}^+_{k}$, and 
$\overline{M}^+_{k}-\overline{N}^+_{k}$. Here $\overline{M}^+_{k}$ and 
$\overline{N}^+_{k}$ are defined in the same fashion
as $M^+_{k}$ and $N^+_{k}$.
In \cite{ACKO} Andrews, Chan, Kim, and Osburn established
$\overline{N}^+_{1}(n) > \overline{N}^+_{1}(n)$. However, it is still only
conjectured that $\overline{N}^+_{k}(n) > \overline{N}^+_{k}(n)$
for all $k\ge 1$ and $n\ge 1$.

Unlike in \cite{GarvanJennings}, here $k=1,2$ as a subscript in $\sptBar{}{n}$ does not 
specify the smallest part being odd or even.

\subsection{The Number of Smallest Parts in Overpartitions with Smallest Part Even}

Next we restrict to overpartitions where the smallest part is even. We denote 
the number of overpartitions
of $n$ with smallest part even by $\overline{p2}(n)$. Thus
$\overline{p2}(n)=4$ since such overpartitions of $4$
are $4$, $\overline{4}$, $2+2$, and $\overline{2}+2$.
In \cite{BLO2} Bringmann, Lovejoy, and Osburn defined 
the associated smallest parts function $\sptBarTwo{}{n}$. As with
$\sptBar{}{n}$, we only include the overpartitions where the smallest part
is not overlined. Thus $\sptBarTwo{}{4}=3$.

We use the $M_2$-rank of an overpartition $\pi$. This rank is given by
\begin{align}
	M_2\mbox{-rank} 
	&= \left\lceil\frac{l(\pi)}{2} \right\rceil - \#(\pi) + \#(\pi_o) - \chi(\pi),
\end{align}
where $l(\pi)$ is the largest part of $\pi$, $\#(\pi)$ is the number of
parts of $\pi$,
$\#(\pi_o)$ is the number of odd non-overlined parts of $\pi$, and
$\chi(\pi) = 1$ if the largest part of $\pi$ is odd and non-overlined and 
$\chi(\pi) = 0$ otherwise. The $M_2$-rank for overpartitions was introduced by 
Lovejoy in \cite{Lovejoy3}. We let $\overline{N2}(m,n)$ denote the number of 
overpartitions of $n$ with $M_2$-rank $m$. 
Lovejoy found the generating function for $\overline{N2}$ is given by
\begin{align}
	\overline{R2}(z,q) 
	&= 
	\sum_{n=0}^\infty\sum_{m=-\infty}^\infty \overline{N2}(m,n)z^mq^n
	= 
	\sum_{n=0}^\infty q^n \frac{\aqprod{-1}{q}{2n}}
		{\aqprod{zq^2}{q^2}{n}\aqprod{z^{-1}q^2}{q^2}{n}}
	\\
	\label{OverpartitionM2Rank}
	&= \frac{\aqprod{-q}{q}{\infty}}{\aqprod{q}{q}{\infty}}
		\SBrackets{1+
			2\sum_{n=1}^\infty\frac{(1-z)(1-z^{-1})(-1)^nq^{n^2+2n}}
				{(1-zq^{2n})(1-z^{-1}q^{2n})}
		}
.
\end{align}

We also use the second residual crank from \cite{BLO2}. 
For an overpartition $\pi$ of $n$ we take 
the crank of the partition $\frac{\pi_e}{2}$ obtained by taking the 
subpartition $\pi_e$, of the even non-overlined parts of 
$\pi$, and halving each part of $\pi_e$.
We let $\overline{M2}(m,n)$ denote the number of overpartitions $\pi$ of $n$  
and such that the partition $\frac{\pi_e}{2}$ has crank $m$. 
Then the generating function for $\overline{M2}$ is given by
\begin{align}\label{OverpartitionResidual2CrankDef}
	\sum_{n=0}^\infty\sum_{m=-\infty}^\infty \overline{M2}(m,n)z^mq^n
	&= \frac{\aqprod{-q}{q}{\infty}\aqprod{q^2}{q^2}{\infty}}
		{\aqprod{q}{q^2}{\infty}\aqprod{zq^2}{q^2}{\infty}\aqprod{z^{-1}q^2}{q^2}{\infty}}.
\end{align}
Again this interpretation  fails for overpartitions 
whose only even non-overlined parts are a single two.

We have the ordinary and symmetrized moments defined by
\begin{align*}
	\overline{N2}_k(n) 
	&= 
	\sum_{m\in\mathbb{Z}} m^k \overline{N2}(m,n),
	\\
	\overline{M2}_k(n) 
	&= 
	\sum_{m\in\mathbb{Z}} m^k \overline{M2}(m,n)		
	,\\
	\overline{\eta 2}_k(n) 
	&= 
	\sum_{m\in\mathbb{Z}}\Bin{m+\Floor{\frac{k-1}{2}}}{k} \overline{N2}(m,n),
	\\
	\overline{\mu 2}_k(n) 
	&= 
	\sum_{m\in\mathbb{Z}} \Bin{m+\Floor{\frac{k-1}{2}}}{k} \overline{M2}(m,n)		
.
\end{align*}
Again these sums are finite sums, the odd moments are zero, and 
\begin{align*}
	\overline{\eta 2}_{2k}(n) 
	&= 
	\frac{1}{(2k)!}
	\sum_{m\in\mathbb{Z}} g_k(m) \overline{N2}(m,n)
	,\\
	\overline{\mu 2}_{2k}(n) 
	&= 
	\frac{1}{(2k)!}
	\sum_{m\in\mathbb{Z}} g_k(m) \overline{M2}(m,n)	
	,\\
	\overline{N2}_{2k}(n) 
	&= 
	\sum_{j=1}^{k} (2j)! S^*(k,j) \overline{\eta 2}_{2j}(n)
	,\\
	\overline{M2}_{2k}(n) 
	&= 
	\sum_{j=1}^{k} (2j)! S^*(k,j) \overline{\mu 2}_{2j}(n)
.
\end{align*}

As with $\sptBar{}{n}$, we define the higher order smallest parts function
$\sptBarTwo{k}{n} = 	\overline{\mu 2}_{2k}(n) - \overline{\eta 2}_{2k}(n)$.
Based on the combinatorial interpretations of Section 3, we have
$\sptBarTwo{}{n} = \overline{\mu 2}_{2k}(n) - \overline{\eta 2}_{2k}(n)$.
In Corollary \ref{CorollaryForSptBar2} we find a generating function for 
$\sptBarTwo{k}{n}$ to be given by
\begin{align*}
	\sum_{n=1}^\infty \sptBarTwo{k}{n} q^n
	&=
	\sum_{n=1}^\infty (\overline{\mu 2}_{2k}(n)- \overline{\eta 2}_{2k}(n)) q^n
	\\
	&=
	\sum_{n_k\ge n_{k-1}\ge \dots \ge n_1 \ge 1}
	\frac{q^{2n_1+2n_2+\dots+2n_k}}{(1-q^{2n_k})^2(1-q^{2n_{k-1}})^2\dots(1-q^{2n_1})^2}
	\frac{\aqprod{-q^{2n_1+1}}{q}{\infty}}{\aqprod{q^{2n_1+1}}{q}{\infty}}
.
\end{align*}
From this generating function, in Corollary \ref{CorollaryMomentInequalities},
we deduce that $\overline{M2}_{2k}(n) > \overline{N2}_{2k}(n)$ 
for $n=2$ and $n\ge 4$. In \cite{Mao2}, Mao derived asymptotics for 
$\overline{N2}_{2k}$ (as well as $\overline{N}_{2k}$), however we do not yet
have asymptotics for $\overline{M2}_{2k}(n)$ nor
$\overline{M2}_{2k}(n) - \overline{N2}_{2k}(n)$.

It is important to note that in \cite{LRS} Larsen, Rust, and Swisher proved a 
stronger result than $\overline{M2}_{2k}(n) > \overline{N2}_{2k}(n)$.
In particular they proved that
\begin{align*}
	\overline{M2}^+_{k}(n) > \overline{N2}^+_{k}(n),
\end{align*}
where $\overline{N2}^+_k(n)$ and $\overline{M2}_k(n)$ are defined similarly
to $N^+_k(n)$ and $M^+_k(n)$. The methods used to handle when the series are 
only over $m\ge 1$ are quite
different than the methods used here. Also they extend a result of Mao 
\cite{Mao} that $\overline{N}_{2k}>\overline{N2}_{2k}$ to
$\overline{N}^+_{k}>\overline{N2}^+_{k}$.

\subsection{The Number of Smallest Parts in Partitions with Smallest Part Even 
and without Repeated Odd Parts}

Lastly we consider partitions with smallest part even and without repeated odd 
parts. We let $p2(n)$ denote the number of such partitions of $n$.
We see $p2(4) = 2$ from the partitions $4$ and $2+2$.
In \cite{ABL} Ahlgren, Bringmann, and Lovejoy defined $\Mspt{}{n}$ 
to be the number of smallest parts in the partitions of $n$
without repeated odd parts and with smallest part even.
We see $\Mspt{}{4}=3$.

We recall the $M_2$-rank of a partition $\pi$ without repeated odd
parts is given by
\begin{align}
	M_2\mbox{-rank} = \left\lceil\frac{l(\pi)}{2} \right\rceil - \#(\pi)
,
\end{align}
where $l(\pi)$ is the largest part of $\pi$ and $\#(\pi)$ is the number of
parts of $\pi$.
The $M_2$-rank was introduced by Berkovich and Garvan
in \cite{BG2}. We let $N2(m,n)$ denote the number of partitions of $n$ with
distinct odd parts 
and $M_2$-rank $m$. By Lovejoy and Osburn \cite{LO2} the generating function 
for $N2(m,n)$, which we further rearrange as in \cite{GarvanJennings} (using 
Watson's transformation), is given by
\begin{align}
	R2(z,q) 
	&= 
	\sum_{n=0}^\infty\sum_{m=-\infty}^\infty N2(m,n)z^mq^n
	= 
	\sum_{n=0}^\infty q^{n^2}\frac{\aqprod{-q}{q^2}{n}}
		{\aqprod{zq^2}{q^2}{n}\aqprod{z^{-1}q^2}{q^2}{n}}
	\\
	&=\frac{\aqprod{-q}{q^2}{\infty}}{\aqprod{q^2}{q^2}{\infty}}
	\SBrackets{1+
		\sum_{n=1}^\infty\frac{(1-z)(1-z^{-1})(-1)^nq^{n(2n+1)}(1+q^{2n})}
			{(1-zq^{2n})(1-z^{-1}q^{2n})}
	}
.
\end{align}

We use another residual crank that was defined in \cite{GarvanJennings}.
For a partition $\pi$ of $n$ with distinct odd parts we take 
the crank of the partition $\frac{\pi_e}{2}$ obtained by taking the 
subpartition $\pi_e$, of the even parts of 
$\pi$, and halving each part of $\pi_e$.
We let $M2(m,n)$ denote the number of partitions $\pi$ of $n$ with distinct odd parts 
and such that the partition $\frac{\pi_e}{2}$ has crank $m$. 
Then the generating function for $M2(m,n)$ is given by
\begin{align}\label{M2Crank}
	\sum_{n=0}^\infty\sum_{m=-\infty}^\infty M2(m,n)z^mq^n
	&= 
		\frac{\aqprod{-q}{q^2}{\infty}\aqprod{q^2}{q^2}{\infty}}
			{\aqprod{zq^2}{q^2}{\infty}\aqprod{z^{-1}q^2}{q^2}{\infty}}.
\end{align}
Of course this interpretation is not quite correct, here it fails for partitions 
with distinct odd parts whose only even parts are a single two.

We have the ordinary and symmetrized moments defined by
\begin{align*}
	N2_k(n) 
	&= 
	\sum_{m\in\mathbb{Z}} m^k N2(m,n)
	,\\
	M2_k(n) 
	&= 
	\sum_{m\in\mathbb{Z}} m^k M2(m,n)
	,\\
	\eta 2_k(n) 
	&= 
	\sum_{m\in\mathbb{Z}} \Bin{m+\Floor{\frac{k-1}{2}}}{k} N2(m,n)
	,\\
	\mu 2_k(n) 
	&= 
	\sum_{m\in\mathbb{Z}} \Bin{m+\Floor{\frac{k-1}{2}}}{k} M2(m,n)
.
\end{align*}
Again these sums are finite sums, the odd moments are zero, and 
\begin{align*}
	\eta 2_{2k}(n) 
	&= 
	\frac{1}{(2k)!}
	\sum_{m\in\mathbb{Z}} g_k(m) N2(m,n),
	\\
	\mu 2_{2k}(n) 
	&=
	\frac{1}{(2k)!} 
	\sum_{m\in\mathbb{Z}} g_k(m) M2(m,n)
	,\\
	N2_{2k}(n) 
	&= 
	\sum_{j=1}^{k} (2j)! S^*(k,j) \eta 2_{2j}(n),
	\\
	M2_{2k}(n) 
	&= 
	\sum_{j=1}^{k} (2j)! S^*(k,j) \mu 2_{2j}(n)
.
\end{align*}

We define the higher order smallest parts function 
$\Mspt{k}{n} = \mu 2_{2k}(n) - 	\eta 2_{2k}(n)$, which in Section 3
we find does agree with $\Mspt{}{n}$ so that 
$\Mspt{1}{n}=\Mspt{}{n}$. In Corollary \ref{CorollaryForM2Spt} we find a 
generating function for $\Mspt{k}{n}$ is
\begin{align*}
	&\sum_{n=1}^\infty \Mspt{k}{n} q^n
	=
	\sum_{n=1}^\infty (\mu 2_{2k}(n)-\eta 2_{2k}(n)) q^n
	\\
	&=
	\sum_{n_k\ge n_{k-1}\ge \dots \ge n_1 \ge 1}
	\frac{q^{2n_1+2n_2+\dots+2n_k}}{(1-q^{2n_k})^2(1-q^{2n_{k-1}})^2\dots(1-q^{2n_1})^2}
	\frac{\aqprod{-q^{2n_1+1}}{q^2}{\infty}}{\aqprod{q^{2n_1+2}}{q^2}{\infty}}
.
\end{align*}
Again we use the generating function in Corollary 
\ref{CorollaryMomentInequalities} to deduce the inequality between ordinary moments,
$M2_{2k}(n) > N2_{2k}(n)$ for $n=2$ and $n\ge 4$.
No one has yet given asymptotics for these rank and crank moments.
Also no one has yet investigated the corresponding 
$M2_{k}^+$ and $N2_{k}^+$. Numerical evidence suggests
$M2_{k}^+ > N2_{k}^+$ for all $k\ge 1$ and $n\ge 4$.

\section{Theorems and Proofs}
   
For $C2(z,q)$, $\overline{C}(z,q)$, and $\overline{C2}(z,q)$ we use that
\begin{align*}
	\frac{\aqprod{q}{q}{\infty}}{\aqprod{zq,z^{-1}q}{q}{\infty}}
	&=
	\frac{1}{\aqprod{q}{q}{\infty}}
	\Parans{ 
		1 
		+ 
		\sum_{n=1}^\infty\frac{(1-z)(1-z^{-1})(-1)^nq^{n(n+1)/2}(1+q^n)}
			{(1-zq^n)(1-z^{-1}q^n)}
	}
,
\end{align*}
this is \cite[equation (7.15)]{Garvan1}. Thus
\begin{align*}
	C2(z,q)
	&=
	\frac{\aqprod{-q}{q^2}{\infty}\aqprod{q^2}{q^2}{\infty}}
	{\aqprod{zq^2}{q^2}{\infty}\aqprod{z^{-1}q^2}{q^2}{\infty}}
	\\
	&=
	\frac{\aqprod{-q}{q^2}{\infty}}{\aqprod{q^2}{q^2}{\infty}}
	\SBrackets{
		1
		+
		\sum_{n=1}^\infty\frac{(1-z)(1-z^{-1})(-1)^nq^{n(n+1)}(1+q^{2n})}
			{(1-zq^{2n})(1-z^{-1}q^{2n})}
	}
	\\
	&=
	\frac{\aqprod{-q}{q^2}{\infty}}{\aqprod{q^2}{q^2}{\infty}}
	\SBrackets{
		1
		+
		\sum_{n=1}^\infty(-1)^nq^{n(n+1)} 
		\Parans{\frac{1-z}{1-zq^{2n}} + \frac{1-z^{-1}}{1-z^{-1}q^{2n}}}
	}
	\\
	&=
	\frac{\aqprod{-q}{q^2}{\infty}}{\aqprod{q^2}{q^2}{\infty}}
	\sum_{n=-\infty}^\infty \frac{(-1)^nq^{n(n+1)}(1-z)}{1-zq^{2n}}
	.
\end{align*}
And similarly we have
\begin{align*}
	\overline{C}(z,q)
	&=
	\frac{\aqprod{-q}{q}{\infty}}{\aqprod{q}{q}{\infty}}
	\sum_{n=-\infty}^\infty \frac{(-1)^nq^{n(n+1)/2}(1-z)}{1-zq^n}
	,\\
	\overline{C2}(z,q)
	&=
	\frac{\aqprod{-q}{q}{\infty}}{\aqprod{q}{q}{\infty}}
	\sum_{n=-\infty}^\infty \frac{(-1)^nq^{n(n+1)}(1-z)}{1-zq^{2n}}
.
\end{align*}

We find similar expressions for the ranks.
\begin{align*}
	R2(z,q)
	&=
	\frac{\aqprod{-q}{q^2}{\infty}}{\aqprod{q^2}{q^2}{\infty}}
	\SBrackets{
		1
		+
		\sum_{n=1}^\infty\frac{(1-z)(1-z^{-1})(-1)^nq^{n(2n+1)}(1+q^{2n})}
			{(1-zq^{2n})(1-z^{-1}q^{2n})}
	}
	\\
	&=
	\frac{\aqprod{-q}{q^2}{\infty}}{\aqprod{q^2}{q^2}{\infty}}
	\SBrackets{
		1
		+
		\sum_{n=1}^\infty(-1)^nq^{n(2n+1)} 
		\Parans{\frac{1-z}{1-zq^{2n}} + \frac{1-z^{-1}}{1-z^{-1}q^{2n}}}
	}
	\\
	&=
	\frac{\aqprod{-q}{q^2}{\infty}}{\aqprod{q^2}{q^2}{\infty}}
	\sum_{n=-\infty}^\infty \frac{(-1)^nq^{n(2n+1)}(1-z)}{1-zq^{2n}}
	.
\end{align*}
Next we have,
\begin{align*}
	\overline{R}(z,q)
	&=
	\frac{\aqprod{-q}{q}{\infty}}{\aqprod{q}{q}{\infty}}
	\SBrackets{
		1
		+
		2\sum_{n=1}^\infty\frac{(1-z)(1-z^{-1})(-1)^nq^{n^2+n}}
			{(1-zq^{n})(1-z^{-1}q^{n})}
	}
	\\
	&=
	\frac{\aqprod{-q}{q}{\infty}}{\aqprod{q}{q}{\infty}}
	\SBrackets{
		1
		+
		2\sum_{n=1}^\infty \frac{(-1)^nq^{n^2+n}}{(1+q^n)} 
		\Parans{\frac{1-z}{1-zq^{n}} + \frac{1-z^{-1}}{1-z^{-1}q^{n}}}
	}
	\\
	&=
	2\frac{\aqprod{-q}{q}{\infty}}{\aqprod{q}{q}{\infty}}
	\sum_{n=-\infty}^\infty \frac{(-1)^nq^{n^2+n}(1-z)}{(1+q^n)(1-zq^{n})}
	.
\end{align*}
Similarly we have
\begin{align*}
	\overline{R2}(z,q)
	&=
	2\frac{\aqprod{-q}{q}{\infty}}{\aqprod{q}{q}{\infty}}
	\sum_{n=-\infty}^\infty \frac{(-1)^nq^{n^2+2n}(1-z)}{(1+q^{2n})(1-zq^{2n})}
	.
\end{align*}

Using that
\begin{align*}
	\Parans{\frac{\partial}{\partial z}}^j \frac{1-z}{1-zq^{n}}
	&=
	\frac{-j!(1-q^n)q^{n(j-1)}}{(1-zq^n)^{j+1}}
,
\end{align*}
we find that the partial derivatives of the crank and rank generating functions
are as follows,
\begin{align*}
	C2^{(j)}(z,q)
	&=
	\Parans{\frac{\partial}{\partial z}}^j C2(z,q)
	=
	\frac{-j!\aqprod{-q}{q^2}{\infty}}{\aqprod{q^2}{q^2}{\infty}}
	\sum_{n\not=0} \frac{(-1)^n q^{n(n-1)+2jn}(1-q^{2n})}{(1-zq^{2n})^{j+1}}
	,
	\\
	R2^{(j)}(z,q)
	&=
	\Parans{\frac{\partial}{\partial z}}^j R2(z,q)
	=
	\frac{-j!\aqprod{-q}{q^2}{\infty}}{\aqprod{q^2}{q^2}{\infty}}
	\sum_{n\not=0} \frac{(-1)^n q^{n(2n-1)+2jn}(1-q^{2n})}{(1-zq^{2n})^{j+1}}
	,
	\\
	\overline{C}^{(j)}(z,q)
	&=
	\Parans{\frac{\partial}{\partial z}}^j \overline{C}(z,q)
	=
	\frac{-j!\aqprod{-q}{q}{\infty}}{\aqprod{q}{q}{\infty}}
	\sum_{n\not=0} \frac{(-1)^n q^{n(n-1)/2+jn}(1-q^{n})}{(1-zq^{n})^{j+1}}
	,
	\\
	\overline{R}^{(j)}(z,q)
	&=
	\Parans{\frac{\partial}{\partial z}}^j \overline{R}(z,q)
	=
	2\frac{-j!\aqprod{-q}{q}{\infty}}{\aqprod{q}{q}{\infty}}
	\sum_{n\not=0} \frac{(-1)^n q^{n^2+jn}(1-q^{n})}{(1+q^n)(1-zq^{n})^{j+1}}
	,
	\\
	\overline{C2}^{(j)}(z,q)
	&=
	\Parans{\frac{\partial}{\partial z}}^j \overline{C2}(z,q)
	=
	\frac{-j!\aqprod{-q}{q}{\infty}}{\aqprod{q}{q}{\infty}}
	\sum_{n\not=0} \frac{(-1)^n q^{n(n-1)+2jn}(1-q^{2n})}{(1-zq^{2n})^{j+1}}
	,
	\\
	\overline{R2}^{(j)}(z,q)
	&=
	\Parans{\frac{\partial}{\partial z}}^j \overline{R2}(z,q)
	=
	2\frac{-j!\aqprod{-q}{q}{\infty}}{\aqprod{q}{q}{\infty}}
	\sum_{n\not=0} \frac{(-1)^n q^{n^2+2jn}(1-q^{2n})}{(1+q^{2n})(1-zq^{2n})^{j+1}}
.
\end{align*}

We collect all expressions for the symmetrized moments in one theorem.
Some of these have been used and proved before in the various papers about 
these moments.
\begin{theorem}\label{TheoremExpressionsForSymmetric}
For all $k\ge 1$
\begin{align*}
	\sum_{n=1}^\infty \mu 2_{2k}(n) q^n
	&=
	\frac{\aqprod{-q}{q^2}{\infty}}{\aqprod{q^2}{q^2}{\infty}}
	\sum_{n\not=0} \frac{(-1)^{n+1} q^{n(n+1)+2kn}}{(1-q^{2n})^{2k}}
	\\
	&=
	\frac{\aqprod{-q}{q^2}{\infty}}{\aqprod{q^2}{q^2}{\infty}}
	\sum_{n\ge 1} \frac{(-1)^{n+1} q^{n(n-1)+2kn}(1+q^{2n})}{(1-q^{2n})^{2k}}
	,
	\\
	\sum_{n=1}^\infty \eta 2_{2k}(n) q^n
	&=
	\frac{\aqprod{-q}{q^2}{\infty}}{\aqprod{q^2}{q^2}{\infty}}
	\sum_{n\not=0} \frac{(-1)^{n+1} q^{n(2n+1)+2kn}}{(1-q^{2n})^{2k}}
	\\
	&=
	\frac{\aqprod{-q}{q^2}{\infty}}{\aqprod{q^2}{q^2}{\infty}}
	\sum_{n\ge 1} \frac{(-1)^{n+1} q^{n(2n-1)+2kn}(1+q^{2n})}{(1-q^{2n})^{2k}}
	,
	\\
	\sum_{n=1}^\infty \overline{\mu}_{2k}(n) q^n
	&=
	\frac{\aqprod{-q}{q}{\infty}}{\aqprod{q}{q}{\infty}}
	\sum_{n\not=0} \frac{(-1)^{n+1} q^{n(n+1)/2+kn}}{(1-q^n)^{2k}}
	\\
	&=
	\frac{\aqprod{-q}{q}{\infty}}{\aqprod{q}{q}{\infty}}
	\sum_{n\ge 1} \frac{(-1)^{n+1} q^{n(n-1)/2+kn}(1+q^{n})}{(1-q^{n})^{2k}}
	,
	\\
	\sum_{n=1}^\infty \overline{\eta}_{2k}(n) q^n
	&=
	2\frac{\aqprod{-q}{q}{\infty}}{\aqprod{q}{q}{\infty}}
	\sum_{n\not=0} \frac{(-1)^{n+1} q^{n^2+n+kn}}{(1+q^n)(1-q^n)^{2k}}
	\\
	&=
	2\frac{\aqprod{-q}{q}{\infty}}{\aqprod{q}{q}{\infty}}
	\sum_{n\ge 1} \frac{(-1)^{n+1} q^{n^2+kn}}{(1-q^n)^{2k}}
	,
	\\
	\sum_{n=1}^\infty \overline{\mu 2}_{2k}(n) q^n
	&=
	\frac{\aqprod{-q}{q}{\infty}}{\aqprod{q}{q}{\infty}}
	\sum_{n\not=0} \frac{(-1)^{n+1} q^{n(n+1)+2kn}}{(1-q^{2n})^{2k}}
	\\
	&=
	\frac{\aqprod{-q}{q}{\infty}}{\aqprod{q}{q}{\infty}}
	\sum_{n\ge 1} \frac{(-1)^{n+1} q^{n(n-1)+2kn}(1+q^{2n})}{(1-q^{2n})^{2k}}
	,
	\\
	\sum_{n=1}^\infty \overline{\eta 2}_{2k}(n) q^n
	&=
	2\frac{\aqprod{-q}{q}{\infty}}{\aqprod{q}{q}{\infty}}
	\sum_{n\not=0} \frac{(-1)^{n+1} q^{n^2+2n+2kn}}{(1+q^{2n})(1-q^{2n})^{2k}}
	\\
	&=
	2\frac{\aqprod{-q}{q}{\infty}}{\aqprod{q}{q}{\infty}}
	\sum_{n\ge 1} \frac{(-1)^{n+1} q^{n^2+2kn}}{(1-q^{2n})^{2k}}
.
\end{align*}
\end{theorem}
\begin{proof}
We follow the proof for a similar expression in Theorem 2 of \cite{Andrews2}.
\begin{align*}	
	\sum_{n=1}^\infty \mu 2_{2k}(n) q^n
	&=
	\frac{1}{(2k)!}\left( \Parans{\frac{\partial}{\partial z}}^{2k} z^{k-1} C2(z,q) \right|_{z=1}
	\\
	&=
	\frac{1}{(2k)!}\sum_{j=0}^{k-1}\Bin{2k}{j}(k-1)\dots(k-j)C2^{2k-j}(1,q)
	\\
	&=
	\frac{\aqprod{-q}{q^2}{\infty}}{\aqprod{q^2}{q^2}{\infty}}
	\sum_{j=0}^{k-1}\Bin{k-1}{j}
	\sum_{n\not=0} \frac{(-1)^{n+1} q^{n(n-1)+2n(2k-j)}(1-q^{2n})}{(1-q^{2n})^{2k-j+1}}
	\\
	&=
	\frac{\aqprod{-q}{q^2}{\infty}}{\aqprod{q^2}{q^2}{\infty}}
	\sum_{n\not=0} \frac{(-1)^{n+1} q^{n(n-1)+4nk}}{(1-q^{2n})^{2k}}
	\sum_{j=0}^{k-1}\Bin{k-1}{j} (q^{-2n}(1-q^{2n}))^j
	\\
	&=
	\frac{\aqprod{-q}{q^2}{\infty}}{\aqprod{q^2}{q^2}{\infty}}
	\sum_{n\not=0} \frac{(-1)^{n+1} q^{n(n-1)+4nk}}{(1-q^{2n})^{2k}}
	\Parans{1+q^{-2n}(1-q^{2n}))}^{k-1}
	\\
	&=
	\frac{\aqprod{-q}{q^2}{\infty}}{\aqprod{q^2}{q^2}{\infty}}
	\sum_{n\not=0} \frac{(-1)^{n+1} q^{n(n+1)+2nk}}{(1-q^{2n})^{2k}}
	\\
	&=
	\frac{\aqprod{-q}{q^2}{\infty}}{\aqprod{q^2}{q^2}{\infty}}
	\sum_{n\ge 1} \frac{(-1)^{n+1} q^{n(n-1)+2kn}(1+q^{2n})}{(1-q^{2n})^{2k}}
.	
\end{align*}
We omit the proofs of the other identities, as they are near identical to the above,
but with $C2(z,q)$ replaced with $R2(z,q)$, $\overline{C}(z,q)$,
$\overline{R}(z,q)$, $\overline{C2}(z,q)$, and $\overline{R2}(z,q)$ respectively.
\end{proof}

We recall two sequences of functions $\alpha_n$ and $\beta_n$ are a Bailey pair
relative to $(a,q)$ if
\begin{align*}
	\beta_n &= \sum_{r=0}^n \frac{\alpha_r}{\aqprod{q}{q}{n-r}\aqprod{aq}{q}{n+r}}
.
\end{align*}
The following is Theorem 3.3 of \cite{Garvan2},
\begin{theorem}\label{TheoremGarvanBailey}
Suppose $\alpha_n$ and $\beta_n$ are a Bailey pair relative to $(1,q)$ and
$\alpha_0=\beta_0=1$,then
\begin{align*}
	&\sum_{n_k\ge n_{k-1}\ge \dots \ge n_1 \ge 1}
	\frac{\aqprod{q}{q}{n_1}^2 q^{n_1+n_2+\dots+n_k}\beta_{n_1}}
		{(1-q^{n_k})^2(1-q^{n_{k-1}})^2\dots(1-q^{n_1})^2}
	\\
	&=
	\sum_{n_k\ge n_{k-1}\ge \dots \ge n_1 \ge 1}
	\frac{q^{n_1+n_2+\dots+n_k}}{(1-q^{n_k})^2(1-q^{n_{k-1}})^2\dots(1-q^{n_1})^2}
	+
	\sum_{r=1}^\infty \frac{q^{kr}\alpha_r}{(1-q^{r})^{2k}}	
.
\end{align*}
\end{theorem}

The following is Corollary 3.4 of Theorem 3.3 from \cite{Garvan2},
\begin{corollary}\label{CorollayBaileyForCranks}
\begin{align*}
	\sum_{n_k\ge n_{k-1}\ge \dots \ge n_1 \ge 1}
	\frac{q^{n_1+n_2+\dots+n_k}}{(1-q^{n_k})^2(1-q^{n_{k-1}})^2\dots(1-q^{n_1})^2}
	&=
	\sum_{n=1}^\infty \frac{(-1)^{n+1}q^{n(n-1)/2+kn}(1+q^{n})}{(1-q^{n})^{2k}}
.
\end{align*}
\end{corollary}

For $\eta 2$ we will use the following.
\begin{corollary}\label{CorollayBaileyForM2Rank}
\begin{align*}
	&\sum_{n_k\ge n_{k-1}\ge \dots \ge n_1 \ge 1}
	\frac{\aqprod{q^2}{q^2}{n_1}q^{2n_1+2n_2+\dots+2n_k}}
		{\aqprod{-q}{q^2}{n_1}(1-q^{2n_k})^2(1-q^{2n_{k-1}})^2\dots(1-q^{2n_1})^2}
	\\
	&=
	\sum_{n_k\ge n_{k-1}\ge \dots \ge n_1 \ge 1}
	\frac{q^{2n_1+2n_2+\dots+2n_k}}{(1-q^{2n_k})^2(1-q^{2n_{k-1}})^2\dots(1-q^{2n_1})^2}
		\\&\qquad
		+
		\sum_{n=1}^\infty \frac{(-1)^{n}q^{n(2n-1)+2kn}(1+q^{2n})}{(1-q^{2n})^{2k}}
\end{align*}
\end{corollary}
\begin{proof}
We have a Bailey pair for $(1,q^2)$ from \cite[page 468]{Slater} given by
\begin{align*}
	\alpha_n &= \PieceTwo{1}{(-1)^nq^{2n^2}(q^n+q^{-n})}{n=0}{n\ge 1}
	\\
	\beta_n &= \frac{1}{\aqprod{-q,q^2}{q^2}{n}}. 
\end{align*}
Applying Theorem \ref{TheoremGarvanBailey} to this Bailey pair gives the
identity.
\end{proof}

For $\overline{\eta}$ we will use the following.
\begin{corollary}\label{CorollayBaileyForOverlineRank}
\begin{align*}
	&\sum_{n_k\ge n_{k-1}\ge \dots \ge n_1 \ge 1}
	\frac{\aqprod{q}{q}{n_1}^2q^{n_1+n_2+\dots+n_k}}
		{\aqprod{q^2}{q^2}{n_1}(1-q^{n_k})^2(1-q^{n_{k-1}})^2\dots(1-q^{n_1})^2}
	\\
	&=
	\sum_{n_k\ge n_{k-1}\ge \dots \ge n_1 \ge 1}
	\frac{q^{n_1+n_2+\dots+n_k}}{(1-q^{n_k})^2(1-q^{n_{k-1}})^2\dots(1-q^{n_1})^2}
	+
	\sum_{n=1}^\infty \frac{(-1)^{n}2q^{n^2+kn}}{(1-q^{n})^{2k}}
.
\end{align*}
\end{corollary}
\begin{proof}
We have a Bailey pair for $(1,q)$ from \cite[page 469]{Slater} given by
\cite{GarvanJennings}, given by
\begin{align*}
	\alpha_n &= \PieceTwo{1}{(-1)^n2q^{n^2}}{n=0}{n\ge 1}
	\\
	\beta_n &= \frac{1}{\aqprod{q^2}{q^2}{n}}. 
\end{align*}
Applying Theorem \ref{TheoremGarvanBailey} to this Bailey pair gives the
identity.
\end{proof}

Lastly we will use the following corollary for $\overline{\eta 2}$ .
\begin{corollary}\label{CorollayBaileyForOverline2Rank}
\begin{align*}
	&\sum_{n_k\ge n_{k-1}\ge \dots \ge n_1 \ge 1}
	\frac{\aqprod{q^2}{q^2}{n_1}^2\aqprod{q}{q^2}{n_1}^2  q^{2n_1+2n_2+\dots+2n_k}}
		{\aqprod{q^2}{q^2}{2n_1}(1-q^{2n_k})^2(1-q^{2n_{k-1}})^2\dots(1-q^{2n_1})^2}
	\\
	&=
	\sum_{n_k\ge n_{k-1}\ge \dots \ge n_1 \ge 1}
	\frac{q^{2n_1+2n_2+\dots+2n_k}}{(1-q^{2n_k})^2(1-q^{2n_{k-1}})^2\dots(1-q^{2n_1})^2}
	+
	\sum_{n=1}^\infty \frac{(-1)^{n}2q^{n^2+2kn}}{(1-q^{2n})^{2k}}
.
\end{align*}
\end{corollary}
\begin{proof}
As in the proof of Theorem 7 of \cite{AW}, we have
\begin{align*}
	\sum_{j=-L}^L \frac{ z^j q^{j^2} } 
	{{\aqprod{q^2}{q^2}{L-j}\aqprod{q^2}{q^2}{L+j}}}
	&=
	\frac{\aqprod{-zq,-q/z}{q^2}{L}}{\aqprod{q^2}{q^2}{2L}}
.
\end{align*}
Setting $z=-1$ gives a Bailey pair relative to $(1,q^2)$ where
$\alpha_n$ and $\beta_n$ are
\begin{align*}
	\alpha_n &= \PieceTwo{1}{(-1)^n2q^{n^2}}{n=0}{n\ge 1}
	\\
	\beta_n &= \frac{\aqprod{q}{q^2}{n}^2}{\aqprod{q^2}{q^2}{2n}}. 
\end{align*}
This Bailey pair is also given as Lemma 2.3 in \cite{BMS}.
Applying Theorem \ref{TheoremGarvanBailey} to this Bailey pair gives the
identity.
\end{proof}

Next we find expressions for $\Mspt{k}{n}$, $\sptBar{k}{n}$, 
and $\sptBarTwo{k}{n}$.
\begin{corollary}\label{CorollaryForM2Spt}
For all $k\ge 1$,
\begin{align*}
	&\sum_{n=1}^\infty \Mspt{k}{n} q^n
	=
	\sum_{n=1}^\infty (\mu 2_{2k}(n)-\eta 2_{2k}(n)) q^n
	\\
	&=
	\sum_{n_k\ge n_{k-1}\ge \dots \ge n_1 \ge 1}
	\frac{q^{2n_1+2n_2+\dots+2n_k}}{(1-q^{2n_k})^2(1-q^{2n_{k-1}})^2\dots(1-q^{2n_1})^2}
	\frac{\aqprod{-q^{2n_1+1}}{q^2}{\infty}}{\aqprod{q^{2n_1+2}}{q^2}{\infty}}
.
\end{align*}
\end{corollary}
\begin{proof}
By Theorem \ref{TheoremExpressionsForSymmetric}, Corollary 
\ref{CorollayBaileyForCranks} with $q$ replaced by $q^2$,
and Corollary \ref{CorollayBaileyForM2Rank}, we have that
\begin{align*}
	&\sum_{n=1}^\infty (\mu 2_{2k}(n)- \eta 2_{2k}(n)) q^n
	\\
	&=
		\frac{\aqprod{-q}{q^2}{\infty}}{\aqprod{q^2}{q^2}{\infty}}
		\sum_{n\ge 1} \frac{(-1)^{n+1} q^{n(n-1)+2kn}(1+q^{2n})}{(1-q^{2n})^{2k}}
		\\&\qquad
		+
		\frac{\aqprod{-q}{q^2}{\infty}}{\aqprod{q^2}{q^2}{\infty}}
		\sum_{n\ge 1} \frac{(-1)^{n} q^{2n^2-n+2kn}(1+q^{2n})}{(1-q^{2n})^{2k}}
	\\
	&=
	\frac{\aqprod{-q}{q^2}{\infty}}{\aqprod{q^2}{q^2}{\infty}}
	\sum_{n_k\ge n_{k-1}\ge \dots \ge n_1 \ge 1}
	\frac{q^{2n_1+2n_2+\dots+2n_k}}{(1-q^{2n_k})^2(1-q^{2n_{k-1}})^2\dots(1-q^{2n_1})^2}
		\\&\qquad 
		+
		\frac{\aqprod{-q}{q^2}{\infty}}{\aqprod{q^2}{q^2}{\infty}}
		\sum_{n\ge 1} \frac{(-1)^{n} q^{2n^2-n+kn}(1+q^{2n})}{(1-q^{2n})^{2k}}
	\\
	&=
	\frac{\aqprod{-q}{q^2}{\infty}}{\aqprod{q^2}{q^2}{\infty}}
	\sum_{n_k\ge n_{k-1}\ge \dots \ge n_1 \ge 1}
	\frac{\aqprod{q^2}{q^2}{n_1} q^{2n_1+2n_2+\dots+2n_k}}
	{\aqprod{-q}{q^2}{n_1}(1-q^{2n_k})^2(1-q^{2n_{k-1}})^2\dots(1-q^{2n_1})^2}
	\\
	&=
	\sum_{n_k\ge n_{k-1}\ge \dots \ge n_1 \ge 1}
	\frac{q^{2n_1+2n_2+\dots+2n_k}}{(1-q^{2n_k})^2(1-q^{2n_{k-1}})^2\dots(1-q^{2n_1})^2}
	\frac{\aqprod{-q^{2n_1+1}}{q^2}{\infty}}{\aqprod{q^{2n_1+2}}{q^2}{\infty}}
.
\end{align*}
\end{proof}

\begin{corollary}\label{CorollaryForSptBar}
For all $k\ge 1$,
\begin{align*}
	\sum_{n=1}^\infty \sptBar{k}{n} q^n
	&=
	\sum_{n=1}^\infty (\overline{\mu}_{2k}(n)- \overline{\eta}_{2k}(n)) q^n
	\\
	&=
	\sum_{n_k\ge n_{k-1}\ge \dots \ge n_1 \ge 1}
	\frac{q^{n_1+n_2+\dots+n_k}}{(1-q^{n_k})^2(1-q^{n_{k-1}})^2\dots(1-q^{n_1})^2}
	\frac{\aqprod{-q^{n_1+1}}{q}{\infty}}{\aqprod{q^{n_1+1}}{q}{\infty}}
.
\end{align*}
\end{corollary}
\begin{proof}
By Theorem \ref{TheoremExpressionsForSymmetric}, Corollary \ref{CorollayBaileyForCranks},
and Corollary \ref{CorollayBaileyForOverlineRank}, we have that
\begin{align*}
	&\sum_{n=1}^\infty (\overline{\mu}_{2k}(n)- \overline{\eta}_{2k}(n)) q^n
	\\
	&=
		\frac{\aqprod{-q}{q}{\infty}}{\aqprod{q}{q}{\infty}}
		\sum_{n\ge 1} \frac{(-1)^{n+1} q^{n(n-1)/2+kn}(1+q^{n})}{(1-q^{n})^{2k}}
		+
		2\frac{\aqprod{-q}{q}{\infty}}{\aqprod{q}{q}{\infty}}
		\sum_{n\ge 1} \frac{(-1)^{n} q^{n^2+kn}}{(1-q^n)^{2k}}
	\\
	&=
	\frac{\aqprod{-q}{q}{\infty}}{\aqprod{q}{q}{\infty}}
		\sum_{n_k\ge n_{k-1}\ge \dots \ge n_1 \ge 1}
		\frac{q^{n_1+n_2+\dots+n_k}}{(1-q^{n_k})^2(1-q^{n_{k-1}})^2\dots(1-q^{n_1})^2}
		\\&\quad
		+
		\frac{\aqprod{-q}{q}{\infty}}{\aqprod{q}{q}{\infty}}
		\sum_{n\ge 1} \frac{(-1)^{n} 2q^{n^2+kn}}{(1-q^n)^{2k}}
	\\
	&=
	\frac{\aqprod{q^2}{q^2}{\infty}}{\aqprod{q}{q}{\infty}^2}
	\sum_{n_k\ge n_{k-1}\ge \dots \ge n_1 \ge 1}
	\frac{\aqprod{q}{q}{n_1}^2q^{n_1+n_2+\dots+n_k}}
	{\aqprod{q^2}{q^2}{n_1}(1-q^{n_k})^2(1-q^{n_{k-1}})^2\dots(1-q^{n_1})^2}
	\\
	&=
	\sum_{n_k\ge n_{k-1}\ge \dots \ge n_1 \ge 1}
	\frac{q^{n_1+n_2+\dots+n_k}}{(1-q^{n_k})^2(1-q^{n_{k-1}})^2\dots(1-q^{n_1})^2}
	\frac{\aqprod{q^{2n_1+2}}{q^2}{\infty}}{\aqprod{q^{n_1+1}}{q}{\infty}^2}
	\\
	&=
	\sum_{n_k\ge n_{k-1}\ge \dots \ge n_1 \ge 1}
	\frac{q^{n_1+n_2+\dots+n_k}}{(1-q^{n_k})^2(1-q^{n_{k-1}})^2\dots(1-q^{n_1})^2}
	\frac{\aqprod{-q^{n_1+1}}{q}{\infty}}{\aqprod{q^{n_1+1}}{q}{\infty}}
.
\end{align*}
\end{proof}

\begin{corollary}\label{CorollaryForSptBar2}
For all $k\ge 1$,
\begin{align*}
	\sum_{n=1}^\infty \sptBarTwo{k}{n} q^n
	&=
	\sum_{n=1}^\infty (\overline{\mu 2}_{2k}(n)- \overline{\eta 2}_{2k}(n)) q^n
	\\
	&=
	\sum_{n_k\ge n_{k-1}\ge \dots \ge n_1 \ge 1}
	\frac{q^{2n_1+2n_2+\dots+2n_k}}{(1-q^{2n_k})^2(1-q^{2n_{k-1}})^2\dots(1-q^{2n_1})^2}
	\frac{\aqprod{-q^{2n_1+1}}{q}{\infty}}{\aqprod{q^{2n_1+1}}{q}{\infty}}
.
\end{align*}
\end{corollary}
\begin{proof}
By Theorem \ref{TheoremExpressionsForSymmetric}, Corollary \ref{CorollayBaileyForCranks}
with $q$ replaced by $q^2$,
and Corollary \ref{CorollayBaileyForOverline2Rank}, we have that
\begin{align*}
	&\sum_{n=1}^\infty (\overline{\mu 2}_{2k}(n)- \overline{\eta 2}_{2k}(n)) q^n
	\\
	&=
		\frac{\aqprod{-q}{q}{\infty}}{\aqprod{q}{q}{\infty}}
		\sum_{n\ge 1} \frac{(-1)^{n+1} q^{n(n-1)+2kn}(1+q^{2n})}{(1-q^{2n})^{2k}}
		+
		2\frac{\aqprod{-q}{q}{\infty}}{\aqprod{q}{q}{\infty}}
		\sum_{n\ge 1} \frac{(-1)^{n} q^{n^2+2kn}}{(1-q^{2n})^{2k}}
	\\
	&=
	\frac{\aqprod{-q}{q}{\infty}}{\aqprod{q}{q}{\infty}}
		\sum_{n_k\ge n_{k-1}\ge \dots \ge n_1 \ge 1}
		\frac{q^{2n_1+2n_2+\dots+2n_k}}{(1-q^{2n_k})^2(1-q^{2n_{k-1}})^2\dots(1-q^{2n_1})^2}
		\\&\quad
		+
		\frac{\aqprod{-q}{q}{\infty}}{\aqprod{q}{q}{\infty}}
		\sum_{n\ge 1} \frac{(-1)^{n} 2q^{n^2+2kn}}{(1-q^{2n})^{2k}}
	\\
	&=
	\frac{\aqprod{-q}{q}{\infty}}{\aqprod{q}{q}{\infty}}
	\sum_{n_k\ge n_{k-1}\ge \dots \ge n_1 \ge 1}
	\frac{\aqprod{q^2}{q^2}{n_1}^2\aqprod{q}{q^2}{n_1}^2  q^{2n_1+2n_2+\dots+2n_k}}
	{\aqprod{q^2}{q^2}{2n_1}(1-q^{2n_k})^2(1-q^{2n_{k-1}})^2\dots(1-q^{2n_1})^2}
	\\
	&=
	\frac{\aqprod{q^2}{q^2}{\infty}}{\aqprod{q}{q}{\infty}^2}
	\sum_{n_k\ge n_{k-1}\ge \dots \ge n_1 \ge 1}
	\frac{\aqprod{q}{q}{2n_1}^2 q^{2n_1+2n_2+\dots+2n_k}}
	{\aqprod{q^2}{q^2}{2n_1}(1-q^{2n_k})^2(1-q^{2n_{k-1}})^2\dots(1-q^{2n_1})^2}
	\\
	&=
	\sum_{n_k\ge n_{k-1}\ge \dots \ge n_1 \ge 1}
	\frac{q^{2n_1+2n_2+\dots+2n_k}}{(1-q^{2n_k})^2(1-q^{2n_{k-1}})^2\dots(1-q^{2n_1})^2}
	\frac{\aqprod{q^{4n_1+2}}{q^2}{\infty}}{\aqprod{q^{2n_1+1}}{q}{\infty}^2}
	\\
	&=
	\sum_{n_k\ge n_{k-1}\ge \dots \ge n_1 \ge 1}
	\frac{q^{2n_1+2n_2+\dots+2n_k}}{(1-q^{2n_k})^2(1-q^{2n_{k-1}})^2\dots(1-q^{2n_1})^2}
	\frac{\aqprod{-q^{2n_1+1}}{q}{\infty}}{\aqprod{q^{2n_1+1}}{q}{\infty}}
.
\end{align*}
\end{proof}

It is now clear that $\Mspt{k}{n} \ge 0$, $\sptBar{k}{n} \ge 0$,
and $\sptBarTwo{k}{n} \ge 0$.
Next we consider inequalities between the ordinary moments.

\begin{corollary}\label{CorollaryMomentInequalities}
Suppose $k\ge 1$. For $n=2$ and $n\ge 4$ we have
\begin{align*}
	M2_{2k}(n) > N2_{2k}(n)
.
\end{align*}
For $n\ge 1$ we have
\begin{align*}
	\overline{M}_{2k}(n) > \overline{N}_{2k}(n)
.
\end{align*}
For $n=2$ and $n\ge 4$ we have
\begin{align*}
	\overline{M2}_{2k}(n) > \overline{N2}_{2k}(n)
.
\end{align*}
\end{corollary}
\begin{proof}
We know
\begin{align*}
	\sum_{n\ge 1} (\mu 2_{2j}(n) - \eta 2_{2j}(n)) q^n
	&=
	\frac{q^{2j}\aqprod{-q^3}{q^2}{\infty}}{(1-q^2)^{2j}\aqprod{q^4}{q^2}{\infty}}
	+\dots
\end{align*}
where the omitted terms also have non-negative coefficients. It is then
apparent that 
\begin{align*}
	\mu 2_{2j}(n) > \eta 2_{2j}(n)
\end{align*}
for $j\ge 1$ and $n \ge 2j+2$. This inequality also holds when $n=2j$, but we
instead have equality at $2j+1$.

However, the $S^*(k,j)$ are integers and are positive for for $1\le j \le k$, thus
\begin{align*}
	M2_{2k}(n) - N2_{2k}(n)
	&=
	\sum_{j=1}^k (2j)! S^*(k,j)(\mu 2_{2j}(n) - \eta 2_{2j}(n))
	\\
	&\ge
	\mu 2_{2}(n) - \eta 2_{2}(n)
	\\
	&> 0 
,
\end{align*}
for $n\ge 4$ and $n=2$. 

Next we have
\begin{align*}
	\sum_{n\ge 1} (\overline{\mu}_{2j}(n) - \overline{\eta}_{2j}(n)) q^n
	&=
	\frac{q^{j}\aqprod{-q^2}{q}{\infty}}{(1-q)^{2j}\aqprod{q^2}{q}{\infty}}
	+\dots
\end{align*}
where the omitted terms also have non-negative coefficients. It is then
apparent that 
\begin{align*}
	\overline{\mu}_{2j}(n) > \overline{\eta}_{2j}(n)
\end{align*}
for $j\ge 1$ and $n \ge j$. Similar to the previous case,
\begin{align*}
	\overline{M}_{2k}(n) - \overline{N}_{2k}(n)
	&=
	\sum_{j=1}^k (2j)! S^*(k,j)(\overline{\mu}_{2j}(n) - \overline{\eta}_{2j}(n))
	\\
	&\ge
	\overline{\mu}_{2}(n) - \overline{\eta}_{2}(n)
	\\
	&> 0 
,
\end{align*}
for $n\ge 1$. 

Last we have
\begin{align*}
	\sum_{n\ge 1} (\overline{\mu 2}_{2j}(n) - \overline{\eta 2}_{2j}(n)) q^n
	&=
	\frac{q^{2j}\aqprod{-q^3}{q}{\infty}}{(1-q^2)^{2j}\aqprod{q^3}{q}{\infty}}
	+\dots
\end{align*}
where the omitted terms also have non-negative coefficients. Thus
\begin{align*}
	\overline{\mu 2}_{2j}(n) > \overline{\eta 2}_{2j}(n)
\end{align*}
for $j\ge 1$ and $n \ge 2j+2$. This inequality also holds when $n=2j$, but we
instead have equality at $2j+1$. As before
\begin{align*}
	\overline{M2}_{2k}(n) - \overline{N2}_{2k}(n)
	&=
	\sum_{j=1}^k (2j)! S^*(k,j)(\overline{\mu 2}_{2j}(n) - \overline{\eta 2}_{2j}(n))
	\\
	&\ge
	\overline{\mu 2}_{2}(n) - \overline{\eta 2}_{2}(n)
	\\
	&> 0 
,
\end{align*}
for $n\ge 4$ and $n=2$. 
\end{proof}

\section{Combinatorial Interpretations}
As in \cite{Garvan2}, for a partition $\pi$ where the different parts are
\begin{align*}
	n_1 < n_2 < \dots < n_m,
\end{align*} 
we have $f_j = f_j(\pi)$ is the frequency of the part $n_j$. 

Thinking of overpartition and partitions without repeated odd parts as
pairs of partitions, we make the following definition. Suppose 
$\vec{\pi} = (\pi_1,\dots,\pi_r)$ is
a vector partition of $n$, then $f^1_j = f^1_j(\vec{\pi}) = f_j(\pi_1)$.
We now view overpartitions as partition pairs where $\pi_2$ is a partition 
into distinct parts, and view partitions with distinct odd parts as partition
pairs where $\pi_1$ has only even parts and $\pi_2$ has only distinct odd parts.
For overpartitions with smallest part even, we use a slightly different idea. 
We view an overpartition with smallest part even as a vector partition
$\vec{\pi} = (\pi_1,\pi_2,\pi_3)$ where $\pi_1$ are the non-overlined even parts, 
$\pi_2$ are the non-overlined odd parts, and $\pi_3$ are the overlined parts.
Furthermore, in all three cases we require the smallest part to only occur in $\pi_1$.
We denote the set of overpartitions with smallest part not overlined by $\SB$, 
the set of partitions with smallest part even and non-repeated odds by
$\STwo$, and the set of overpartitions with smallest part even and not
overlined by $\SBTwo$.

We note that
\begin{align*}
	\Mspt{}{n} &= \sum_{\vec{\pi}\in\STwo, |\vec{\pi}|=n  } f_1^1(\vec{\pi}),
	\\
	\sptBar{}{n} &= \sum_{\vec{\pi}\in\SB, |\vec{\pi}|=n  } f_1^1(\vec{\pi}),
	\\
	\sptBarTwo{}{n} &= \sum_{\vec{\pi}\in\SBTwo, |\vec{\pi}|=n  } f_1^1(\vec{\pi}).
\end{align*}

For $k\ge 1$ we extend the weight $\omega_k$ of \cite{Garvan2}, 
for a partition pair $\vec{\pi} = (\pi_1,\pi_2)$ or vector partition
$\vec{\pi} = (\pi_1,\pi_2,\pi_3)$
we let 
$\omega_k(\vec{\pi}) = \omega_k(\pi_1)$. That is,
\begin{align*}
	\omega_k(\vec{\pi})
	&=
	\sum_{\substack{m_1+m_2+\dots+m_r = k\\ 1\le r \le k} }
	\Bin{f^1_1+m_1-1}{2m_1-1}
	\\&\qquad\times
	\sum_{2\le j_2 < j_3 < \dots < j_r} 
	\Bin{f^1_{j_2}+m_2}{2m_2}\Bin{f^1_{j_3}+m_3}{2m_3}\dots\Bin{f^1_{j_r}+m_r}{2m_r}
.
\end{align*}

\begin{theorem}
For all $k\ge 1$ and $n\ge 1$ we have
\begin{align*}
	\Mspt{k}{n} &= \sum_{\vec{\pi}\in\STwo} \omega_k(\vec{\pi}),
	\\
	\sptBar{k}{n} &= \sum_{\vec{\pi}\in\SB} \omega_k(\vec{\pi}),
	\\
	\sptBarTwo{k}{n} &= \sum_{\vec{\pi}\in\SBTwo} \omega_k(\vec{\pi})
.
\end{align*}
\end{theorem}
\begin{proof}
The proof is near identical as that of Theorem 5.6 of \cite{Garvan2},
the only difficulty being how to write out the general case. We will fully write out 
the case when $k=3$ for $\sptBar{k}{n}$, go over
the case of $k=4$ for $\Mspt{k}{n}$,  and 
explain the procedure for general $k$ which will then be clear.

We use
\begin{align*}
	\sum_{n=j}^\infty\Bin{n+j-1}{2j-1}x^n 
	&= 
	\frac{x^j}{(1-x)^{2j}}
	,\\	
	\sum_{n=j}^\infty\Bin{n+j}{2j}x^n 
	&= 
	\frac{x^j}{(1-x)^{2j+1}}
.
\end{align*}

For the $k=3$ case for $\sptBar{k}{n}$, we have
\begin{align*}
	&\sum_{n=1}^\infty (\overline{\mu}_6(n)-\overline{\eta}_6(n))q^n
	\\
	&=
	\sum_{1\le m\le k\le n} \frac{q^{m+k+n}\aqprod{-q^{m+1}}{q}{\infty}}
		{(1-q^m)^2(1-q^k)^2(1-q^n)^2\aqprod{q^{m+1}}{q}{\infty}}
	\\
	&=
	\sum_{1\le m=k=n}
	+\sum_{1\le m=k<n} 
	+\sum_{1\le m<k=n}
		\\&\quad
		+\sum_{1\le m<k<n} 
		\Parans{	
		\frac{q^{m+k+n}\aqprod{-q^{m+1}}{q}{\infty}}
		{(1-q^m)^2(1-q^k)^2(1-q^n)^2\aqprod{q^{m+1}}{q}{\infty}}
		}
	\\
	&=
		\sum_{1\le m}\frac{q^{3m}}{(1-q^m)^6}
		\aqprod{-q^{m+1}}{q}{\infty}
		\prod_{i>m}\frac{1}{1-q^i}
		+
		\\&\quad
		\sum_{1\le m<n} 
		\frac{q^{2m}}{(1-q^m)^4}\frac{q^{n}}{(1-q^n)^3}
		\aqprod{-q^{m+1}}{q}{\infty}
		\prod_{\substack{i>m\\i\not=n}}\frac{1}{1-q^i}
		\\&\quad
		+
		\sum_{1\le m<k}
		\frac{q^{m}}{(1-q^m)^2}\frac{q^{2k}}{(1-q^k)^5}
		\aqprod{-q^{m+1}}{q}{\infty}
		\prod_{\substack{i>m\\i\not=k}}\frac{1}{1-q^i}
		\\&\quad
		+
		\sum_{1\le m<k<n} 	
		\frac{q^{m}}{(1-q^m)^2}\frac{q^{k}}{(1-q^k)^3}\frac{q^{n}}{(1-q^n)^3}
		\aqprod{-q^{m+1}}{q}{\infty}
		\prod_{\substack{i>m\\i\not=k,n}}\frac{1}{1-q^i}
	\\
	&=
		\sum_{1\le m}
		\sum_{f_1=3}^\infty \Bin{f_1+3-1}{6-1} q^{mf_1}
		\aqprod{-q^{m+1}}{q}{\infty}
		\prod_{i>m}\frac{1}{1-q^i}
		\\&\quad
		+
		\sum_{1\le m<n} 
		\sum_{f_1=2}^\infty \Bin{f_1+2-1}{4-1} q^{mf_1}
		\sum_{f_{j_2}=1}^\infty \Bin{f_{j_2}+1}{2} q^{nf_{j_2}}
		\aqprod{-q^{m+1}}{q}{\infty}
		\prod_{\substack{i>m\\i\not=n}}\frac{1}{1-q^i}
		\\&\quad
		+
		\sum_{1\le m<k}
		\sum_{f_1=1}^\infty \Bin{f_1+1-1}{2-1} q^{mf_1}
		\sum_{f_{j_2}=2}^\infty \Bin{f_{j_2}+2}{4} q^{kf_{j_2}}
		\aqprod{-q^{m+1}}{q}{\infty}
		\prod_{\substack{i>m\\i\not=k}}\frac{1}{1-q^i}
		\\&\quad
		+
		\sum_{1\le m<k<n} 	
		\sum_{f_1=1}^\infty \Bin{f_1+1-1}{2-1} q^{mf_1}
		\sum_{f_{j_2}=1}^\infty \Bin{f_{j_2}+1}{2} q^{kf_{j_2}}
		\sum_{f_{j_3}=1}^\infty \Bin{f_{j_3}+1}{2} q^{nf_{j_3}}
			\\&\qquad\qquad\times
			\aqprod{-q^{m+1}}{q}{\infty}
			\prod_{\substack{i>m\\i\not=k,n}}\frac{1}{1-q^i}
.
\end{align*}
The set of the $4$ compositions of $3$ is
$A=\{(3),(2,1),(1,2),(1,1,1)\}$, thus we have
\begin{align*}
	&\sum_{n=1}^\infty (\overline{\mu}_6(n)-\overline{\eta}_6(n))q^n
	\\
	&=
	\sum_{(m_1,\dots,m_r)=\vec{m}\in A}
	\sum_{1\le n_1 < n_{j_2} < \dots < n_{j_r}}	
	\sum_{f_1=m_1}^\infty\sum_{f_{j_2}=m_2}^\infty\dots\sum_{f_{j_r}=m_r}^\infty
	\Bin{f_1+m_1-1}{2m_1-1}
	\nonumber\\& \qquad \times
			\Bin{f_{j_2}+m_2}{2m_2}\dots\Bin{f_{j_r}+m_r}{2m_r}
			q^{n_1f_1 + n_{j_2}f_{j_2} +\dots + n_{j_r}f_{j_r}}
			\aqprod{-q^{n_1+1}}{q}{\infty}
			\nonumber\\& \qquad \times
			\prod_{\substack{i>n_1\\i\not\in\{n_{j_2},\dots,n_{j_r} \}} }
			\frac{1}{1-q^{i}}
.
\end{align*}
This we recognize as the generating function for partition pairs 
$\vec{\pi}=(\pi_1,\pi_2)\in\SB$ counted according to the weight $\omega_3$. This is the
generating function obtained by summing according to the smallest part
of $\pi_1$ being $n_1$ with frequency $f_1$.

For the $k=4$ case for $\Mspt{k}{n}$, we have
\begin{align*}
	&\sum_{n=1}^\infty (\mu 2_8(n) - \eta 2_8(n))q^n
	\\
	&=
	\sum_{1\le m \le j \le k \le n}
	\frac{q^{2m+2j+2k+2n}\aqprod{-q^{2m+1}}{q^2}{\infty}}
	{(1-q^{2m})^2(1-q^{2j})^2(1-q^{2k})^2(1-q^{2n})^2\aqprod{q^{2m+2}}{q^2}{\infty}}	
	\\
	&=
		\sum_{1\le m = j = k = n}
		+\sum_{1\le m = j = k < n}
		+\sum_{1\le m = j < k = n}
		+\sum_{1\le m = j < k < n}
		+\sum_{1\le m < j = k = n}
		+\sum_{1\le m < j = k < n}
		\\&\quad
		+\sum_{1\le m < j < k = n}
		+\sum_{1\le m < j < k < n}
		\frac{q^{2m+2j+2k+2n}\aqprod{-q^{2m+1}}{q^2}{\infty}}
		{(1-q^{2m})^2(1-q^{2j})^2(1-q^{2k})^2(1-q^{2n})^2\aqprod{q^{2m+2}}{q^2}{\infty}}		
	\\
	&=
		\sum_{1\le m} \frac{q^{8m}}{(1-q^{2m})^8}
			\aqprod{-q^{2m+1}}{q^2}{\infty}
			\prod_{i>m}\frac{1}{1-q^{2i}}		
		\\&\quad
		+\sum_{1\le m < n}\frac{q^{6m}}{(1-q^{2m})^6}\frac{q^{2n}}{(1-q^{2n})^3}
			\aqprod{-q^{2m+1}}{q^2}{\infty}
			\prod_{\substack{i>m\\i\not=n}}\frac{1}{1-q^{2i}}
		\\&\quad
		+\sum_{1\le m < k}\frac{q^{4m}}{(1-q^{2m})^4}\frac{q^{4k}}{(1-q^{2k})^5}
			\aqprod{-q^{2m+1}}{q^2}{\infty}
			\prod_{\substack{i>m\\i\not=k}}\frac{1}{1-q^{2i}}
		\\&\quad
		+\sum_{1\le m < k < n}\frac{q^{4m}}{(1-q^{2m})^4}\frac{q^{2k}}{(1-q^{2k})^3}
			\frac{q^{2n}}{(1-q^{2n})^3}
			\aqprod{-q^{2m+1}}{q^2}{\infty}
			\prod_{\substack{i>m\\i\not=k,n}}\frac{1}{1-q^{2i}}
		\\&\quad
		+\sum_{1\le m < j}\frac{q^{2m}}{(1-q^{2m})^2}\frac{q^{6j}}{(1-q^{2j})^7}
			\aqprod{-q^{2m+1}}{q^2}{\infty}
			\prod_{\substack{i>m\\i\not=j}}\frac{1}{1-q^{2i}}
		\\&\quad
		+\sum_{1\le m < j < n}\frac{q^{2m}}{(1-q^{2m})^2}\frac{q^{4j}}{(1-q^{2j})^5}
			\frac{q^{2n}}{(1-q^{2n})^3}
			\aqprod{-q^{2m+1}}{q^2}{\infty}
			\prod_{\substack{i>m\\i\not=j,n}}\frac{1}{1-q^{2i}}
		\\&\quad
		+\sum_{1\le m < j < k}\frac{q^{2m}}{(1-q^{2m})^2}\frac{q^{2j}}{(1-q^{2j})^3}
			\frac{q^{4k}}{(1-q^{2k})^5}
			\aqprod{-q^{2m+1}}{q^2}{\infty}
			\prod_{\substack{i>m\\i\not=j,k}}\frac{1}{1-q^{2i}}
		\\&\quad
		+\sum_{1\le m < j < k < n}\frac{q^{2m}}{(1-q^{2m})^2}\frac{q^{2j}}{(1-q^{2j})^3}
			\frac{q^{2k}}{(1-q^{2k})^3}\frac{q^{2n}}{(1-q^{2n})^3}
			\aqprod{-q^{2m+1}}{q^2}{\infty}
			\prod_{\substack{i>m\\i\not=j,k,n}}\frac{1}{1-q^{2i}}
.
\end{align*}
In order, the above eight terms correspond to the compositions of $4$:
$(4)$, $(3,1)$, $(2,2)$, $(2,1,1)$, $(1,3)$, $(1,2,1)$, $(1,1,2)$, 
$(1,1,1,1)$.

Thus for each composition $m_1+\dots+m_r = 4$ we have a sum of the form:
\begin{align}\label{GeneralTermK4}
	&\sum_{1\le n_1 < n_{j_2} < \dots < n_{j_r}}
	\frac{q^{2n_1m_1}}{(1-q^{2n_1})^{2m_1}}
	\frac{q^{2n_2m_2}}{(1-q^{2n_2})^{2m_2+1}}
	\dots
	\frac{q^{2n_{j_r}m_r}}{(1-q^{2n_{j_r}})^{2m_r+1}}
	\aqprod{-q^{2n_1+1}}{q^2}{\infty}
		\\&\quad\times
		\prod_{\substack{i>n_1\\i\not\in\{n_{j_2},\dots,n_{j_r} \}} }
			\frac{1}{1-q^{2i}}
	\nonumber\\	
	&=
		\sum_{1\le n_1 < n_{j_2} < \dots < n_{j_r}}	
		\sum_{f_1=m_1}^\infty\sum_{f_{j_2}=m_2}^\infty\dots\sum_{f_{j_r}=m_r}^\infty
		\Bin{f_1+m_1-1}{2m_1-1}
		\nonumber\\& \qquad \times
			\Bin{f_{j_2}+m_2}{2m_2}\dots\Bin{f_{j_r}+m_r}{2m_r}
			q^{2n_1f_1 + 2n_{j_2}f_{j_2} +\dots + 2n_{j_r}f_{j_r}}
			\aqprod{-q^{2n_1+1}}{q^2}{\infty}
			\nonumber\\& \qquad \times
			\prod_{\substack{i>n_1\\i\not\in\{n_{j_2},\dots,n_{j_r} \}} }
			\frac{1}{1-q^{2i}}
.
\end{align}
Noting the $f_{j_i}$ correspond to the frequencies of certain even parts, we see 
summing (\ref{GeneralTermK4}) over all compositions of $4$ yields the generating function for 
partitions without repeated odd parts and smallest part even written as a partition pair
$(\pi_1,\pi_2)\in S2$,  counted according
to the weight $\omega_4$. This is the generating function given
by summing according to the smallest part being $2n_1$ with frequency $f_1$.

For general $k$, we take the expression in Corollary \ref{CorollaryForM2Spt},
\ref{CorollaryForSptBar}, or \ref{CorollaryForSptBar2}
and break it into
$2^{k-1}$ sums by turning the index bounds into $=$ or $<$. These correspond to
the $2^{k-1}$ compositions of $k$. The sum with index bounds
$n_1 \square_1 n_2 \square_2 \dots \square_{r-1} n_r$ where each $\square_i$ is
either ``$=$'' or ``$<$'' corresponds to the composition
$\allowbreak{(1 \triangle_1 1 \triangle_2 \dots \triangle_{r-1} 1)}$
where $\triangle_i$ is ``$+$'' if $\square_i$ is ``$=$'' and $\triangle_i$ is ``$,$'' if 
$\square_i$ is ``$<$''.

For $\Mspt{k}{n}$ the sum corresponding to the fixed composition
$m_1+m_2+\dots+m_r = k$ is then rewritten as
\begin{align}\label{GeneralTerm}
	&\sum_{1\le n_1 < n_{j_2} < \dots < n_{j_r}}	
	\sum_{f_1=m_1}^\infty\sum_{f_{j_2}=m_2}^\infty\dots\sum_{f_{j_r}=m_r}^\infty
	\Bin{f_1+m_1-1}{2m_1-1}\Bin{f_{j_2}+m_2}{2m_2}\dots\Bin{f_{j_r}+m_r}{2m_r}
	\nonumber\\& \qquad\qquad \times
	q^{2n_1f_1 + 2n_{j_2}f_{j_2} +\dots + 2n_{j_r}f_{j_r}}
	\aqprod{-q^{2n_1+1}}{q^2}{\infty}
	\prod_{\substack{i>n_1\\i\not\in\{n_{j_2},\dots,n_{j_r} \}} }
		\frac{1}{1-q^{2i}}
.
\end{align}
Thus on the one hand summing (\ref{GeneralTerm}) over all compositions of $k$
gives $\sum_{n=1}^\infty (\mu 2_{2k}(n) - \eta 2_{2k}(n))q^n$, 
but also this is $\sum_{n=1}^\infty q^n\sum_{\vec{\pi}\in\STwo}\omega_k(\vec{\pi})$.
 
For $\sptBar{k}{n}$, the general case follows the same idea, but differs in that
the term for a fixed composition $m_1+\dots+m_r$ of $k$ is
\begin{align*}
	&\sum_{1\le n_1 < n_{j_2} < \dots < n_{j_r}}	
	\sum_{f_1=m_1}^\infty\sum_{f_{j_2}=m_2}^\infty\dots\sum_{f_{j_r}=m_r}^\infty
	\Bin{f_1+m_1-1}{2m_1-1}\Bin{f_{j_2}+m_2}{2m_2}\dots\Bin{f_{j_r}+m_r}{2m_r}
	\nonumber\\& \qquad\qquad \times
			q^{n_1f_1 + n_{j_2}f_{j_2} +\dots + n_{j_r}f_{j_r}}
			\aqprod{-q^{n_1+1}}{q}{\infty}
			\prod_{\substack{i>n_1\\i\not\in\{n_{j_2},\dots,n_{j_r} \}} }
			\frac{1}{1-q^{i}}
.
\end{align*}

Lastly, for $\sptBarTwo{k}{n}$ the general term is
\begin{align*}
	&
	\sum_{1\le n_1 < n_{j_2} < \dots < n_{j_r}}	
	\sum_{f_1=m_1}^\infty\sum_{f_{j_2}=m_2}^\infty\dots\sum_{f_{j_r}=m_r}^\infty
	\Bin{f_1+m_1-1}{2m_1-1}\Bin{f_{j_2}+m_2}{2m_2}\dots\Bin{f_{j_r}+m_r}{2m_r}
	\nonumber\\& \qquad\qquad \times
			q^{2n_1f_1 + 2n_{j_2}f_{j_2} +\dots + 2n_{j_r}f_{j_r}}
			\aqprod{-q^{2n_1+1}}{q}{\infty}
			\prod_{\substack{i>2n_1\\i\not\in\{2n_{j_2},\dots,2n_{j_r} \}} }
			\frac{1}{1-q^{i}}
.
\end{align*}
This finishes the proof.
\end{proof}

To help illustrate what is being counted, we give the details to show
$\sptBar{2}{4}=16$. Using the overpartitions of $4$ listed in the introduction,
we see the partition pairs of $4$ from $\SB$ are
$(4,\emptyset)$, $(3+1,\emptyset)$, $(1,3)$, $(2+2,\emptyset)$, 
$(2+1+1,\emptyset)$, $(1+1,2)$, and $(1+1+1+1,\emptyset)$. Since the 
compositions of $2$ are just $2$ and $1+1$ and for these partition
pairs we have $f^1_j=0$ for $j>2$, the weight $\omega_2$ reduces to
\begin{align*}
	\omega_2(\vec{\pi}) &=
	\Bin{f_1^1(\vec{\pi})+1}{3} 
	+ f_1^1(\vec{\pi})\Bin{f_2^1(\vec{\pi})+1}{2}
.
\end{align*}
For larger values of $n$, $\sptBar{2}{n}$ and $\omega_2(\vec{\pi})$ would
be slightly more complicated as there would be partition pairs with 
$f^1_j\not=0$ for $j$ past $1$ and $2$. We collect the information for each 
partition pair in the following table.
$$
\begin{array}{c|c|c|c}
\mbox{$\SB$-partition pair} & f_1^1   & f_1^2  & \omega_2 \\
\hline
(4,\emptyset) 			& 1 & 0 & 0 \\
(3+1,\emptyset) 		& 1 & 1 & 1 \\
(1,3) 					& 1 & 0 & 0 \\
(2+2,\emptyset) 		& 2 & 0 & 1 \\
(2+1+1,\emptyset) 	& 2 & 1 & 3 \\
(1+1,2) 					& 2 & 0 & 1 \\
(1+1+1+1,\emptyset) 	& 4 & 0 & 10
\end{array}
$$

\section{Congruences for $\sptBar{2}{n}$}
It appears that these higher order spt functions satisfy various congruences. 
We prove two of them.

\begin{theorem}
For $n\ge 0$,
\begin{align*}
	\sptBar{2}{5n+1}\equiv 0 \pmod{5}
	,\\
	\sptBar{2}{5n+3}\equiv 0 \pmod{5}
.
\end{align*}
\end{theorem}
\begin{proof}
We have
\begin{align}\label{CogruencesEq1}
	\sptBar{2}{n}
	&=
	\frac{1}{24}\Parans{
		\overline{M}_4(n)-\overline{M}_2(n)-\overline{N}_4(n)+\overline{N}_2(n)}
.
\end{align}
Reducing equation (3.1) of \cite{BLO2} modulo 5 gives
\begin{align*}
	\overline{N}_4(n)
	&\equiv
	(2n+4)\overline{N}_2(n)+(2n+2)\overline{M}_2(n)+\overline{M}_4(n)
	+2n\overline{M2}_2(n)
	\pmod{5}
,
\end{align*}
so (\ref{CogruencesEq1}) becomes
\begin{align}\label{CogruencesEq2}
	\sptBar{2}{n}
	&\equiv
	(3+2n)\overline{M}_2(n)+2n\overline{M2}_2(n)+(3+2n)\overline{N}_2(n)
	\pmod{5}
.
\end{align}

The following are equation (4.4) and an equation out of the proof of 
Theorem 3.1 of \cite{BLO2}:
\begin{align}
	\label{CogruencesEq3}
	(2n^2+n+2)\overline{M}_2(n) + (n^2+4n+2)\overline{M2}_2(n)
	&\equiv 
	0 \pmod{5}
	,\\
	\label{CogruencesEq4}
	(4n^2+n)\overline{M}_2(n) + (4n+4)\overline{M2}_2(n) + (3n^2+2)\overline{N}_2(n)
	&\equiv 
	0 \pmod{5}
.
\end{align}
In (\ref{CogruencesEq3}) we replace $n$ by $5n+1$ and in (\ref{CogruencesEq4}) we
replace $n$ by $5n+3$ to get
\begin{align*}
	\overline{M2}_2(5n+1)
	&\equiv 
	0 \pmod{5}
	,\\
	4\overline{M}_2(5n+3) + \overline{M2}_2(5n+3) + 4\overline{N}_2(5n+3)
	&\equiv 
	0 \pmod{5}
.
\end{align*}
With (\ref{CogruencesEq2}) we then have
\begin{align*}
	\sptBar{2}{5n+1} &\equiv 2\overline{M2}_2(5n+1) \equiv 0\pmod{5}
	,\\
	\sptBar{2}{5n+3} 
	&\equiv 
	4\overline{M}_2(5n+3)
	+4\overline{N}_2(5n+3)
	+\overline{M2}_2(5n+3) \equiv 0\pmod{5}
.
\end{align*}
\end{proof}

Congruences for $\Mspt{2}{n}$ will be handled in a future paper.

\section{Remarks}
In \cite{DixitYee} Dixit and Yee also generalized the spt function to $\mbox{Spt}_j$
and generalized the higher order spt-function
$\mbox{spt}_{k}$ to $_{j}\mbox{spt}_{k}$. They used
\begin{align*}
	\mbox{Spt}_j(n) &= \frac{1}{2} M_2(n) - {\frac{1}{2}} {_{j+1}N_2(n)},
	\\
	_{j}\mbox{spt}_{k}(n) &= _{j}\mu_{2k}(n) - _{j+1}\mu_{2k}(n),
\end{align*}
where
\begin{align*}
	_{j}N_{k}(n) &= \sum_{m\in\mathbb{Z}} m^k N_j(m,n),
	\\
	_{j}\mu_k(n) &= \sum_{m\in\mathbb{Z}} \Bin{m+\Floor{\frac{k-1}{2}}}{k} N_j(m,n),
\end{align*}
and $N_j(m,n)$ is the number of partitions of $n$ with at least
$j-1$ successive Durfee squares whose $j$-rank is $m$. It may be possible to
work out generalizations of this form for the three
spt functions we have investigated here.

It is worth mentioning that it is not $\overline{R2}$ and $\overline{C2}$
that were used in \cite{GarvanJennings} to reprove certain congruences
satisfied by $\sptBarTwo{}{n}$. However, the methods in that paper can
be used with $\overline{R2}$ and $\overline{C2}$ to prove the congruences 
$\sptBarTwo{}{3n}\equiv\sptBarTwo{}{3n+1}\equiv 0 \pmod{3}$. 
Yet those methods do no work
to prove the congruence $\sptBarTwo{}{5n+3}\equiv 0 \pmod{5}$ with
$\overline{R2}$ and $\overline{C2}$.

\section{References}

\bibliographystyle{elsarticle-num}

\end{document}